    \documentclass[reqno]{amsproc}
    \usepackage{amsmath,amssymb,amsthm,amsfonts,latexsym}
   \usepackage{color}
   \definecolor{Light}{gray}{.75}

   \theoremstyle{plain}
   \newtheorem{thm}{Theorem}
   \newtheorem{pro}[thm]{Proposition}
   
   \newtheorem{cor}[thm]{Corollary}
   \newtheorem{fact}[thm]{Fact}
   \newtheorem{try}[thm]{Triviality}
   
   \newtheorem*{ass}{{Advice}}

   \theoremstyle{definition}

   \theoremstyle{remark}
   \newtheorem{rem}[thm]{{\it Remark}}

   \def\tlabel{\label}

\DeclareMathOperator{\okr}{{\stackrel{{\scriptscriptstyle{\mathsf{def}}}}{=}}}
\DeclareMathOperator{\D}{d\!}
\DeclareMathOperator{\E}{e} \DeclareMathOperator{\I}{i}
\DeclareMathOperator{\RE}{\mathfrak{Re}}
\DeclareMathOperator{\IM}{\mathfrak{Im}}
\DeclareMathOperator{\lin}{lin}
\DeclareMathOperator{\clolin}{clolin}

   \DeclareMathOperator{\supp}{supp}
\def\dz#1{\mathcal D({#1})}

\def\funk#1#2#3{#1\colon#2\to#3}
\def\Ge{\geqslant}
\def\gw{^*}
\def\is#1#2{\langle#1,#2\rangle}

\def\jd#1{\mathcal N(#1)}

\def\Le{\leqslant}

\def\liczp#1{{${#1}^{\text {\rm o}}$}}

\def\nic{\varnothing}

\def\npo#1{^{(#1)}}
\def\ob#1{\mathcal R(#1)}

\def\res#1{|_{#1}}
\def\rres#1{\!\!\upharpoonright_{#1}}

\def\sbar#1{\,\overline{\!#1}}
\def\spek#1{{\rm sp}(#1)}

\def\ulamek#1#2{\mbox{\normalfont$\frac{#1}{#2}$}}

\def\zb#1#2{\{{#1}\colon\ {#2}\}}
\def\aac{\mathcal A}
\def\bbc{\mathcal B}
\def\ccc{\mathcal C}
\def\ddc{\mathcal D}

\def\hhc{\mathcal H}
\def\kkc{\mathcal K}

\def\llc{\mathcal L}
\def\ppc{\mathcal P}

\def\qqc{\mathcal Q}
\def\ssc{\mathcal S}
\def\uuc{\mathcal U}
\def\xxc{\mathcal X}

\def\ccb{\mathbb C}

\def\rrb{\mathbb R}

\def\zzb{\mathbb Z}

\def\bbs{\boldsymbol B}

\begin{document}

   \title[Normals, subnormals and an open question]{Normals, subnormals and an open question   }
   \author[F.H. Szafraniec]{Franciszek Hugon Szafraniec}
   \address{Instytut        Matematyki,         Uniwersytet
   Jagiello\'nski, ul. \L ojasiewicza 6, 30 348 Krak\'ow, Poland}
   \email{umszafra@cyf-kr.edu.pl}
   \thanks{This work was  supported  by the MNiSzW grant
N201 026 32/1350.}
   \dedicatory{To Zolt\'an Sebesty\'en for his 65th birthday
anniversary}
   \begin{abstract}
   An acute look at \underbar{basic} facts concerning
\underbar{unbounded} subnormal operators is taken here. These
operators have the richest structure and are the most exciting
among the whole family of beneficiaries of the normal ones.
Therefore, the latter must necessarily be taken into account as
the reference point for any exposition of subnormality. So as to
make the presentation more appealing a kind of comparative survey
of the bounded and unbounded case has been set forth.

\noindent This piece of writing serves rather as a practical guide
to this largely impenetrable territory than an exhausting report.
   \end{abstract}
   \maketitle


    We begin with bounded operators pointing out those well known
properties of normal and subnormal operators, which in unbounded
case become much more complex. Then we are going to show how the
situation looks like for their unbounded counterparts. The
distinguished example of the creation operator coming from the
quantum harmonic oscillator crowns the theory. Finally we discuss
an open question, one of those which seem to be pretty much
intriguing and hopefully inspiring.

   By an unbounded operator we mean a not necessarily bounded one,
nevertheless it is always considered to be densely defined, always
in a complex Hilbert space. If we want to emphasis an operator to
be everywhere defined we say it is \underbar{on}, otherwise we say
it is \underbar{in}. Unconventionally though suggestively,
$\bbs(\hhc)$ denotes all the bounded operators on $\hhc$. If $A$
is an operator, then $\dz A$, $\jd A$ and $\ob A$ stands for its
domain, kernel(null space) and range respectively; if $A$ is
closable, its closure is denoted by $\sbar A$.

 Let us mention some books where unbounded normal operators are
treated: they are \cite{bir}, \cite[Chapter XV, Section 12]{die},
\cite{nagy42} and \cite{wei}. To bounded subnormal operators the
book \cite{con} is totally devoted.

Despite the ambitious plan the topics presented here are a rather
selective. Also some of the arguments used in the proofs have to
be extended. The material, though still developing, is sizeable
enough to cover a large monograph; this is the project
\cite{book_sub} already in progress.

   \section*{The haven of tranquility of bounded operators}
   \subsection*{Foremost topics of normality}
   \subsubsection*{{Normal operators: around the definition}}

   An operator $N\in\bbs(\hhc)$ is said to be {\em normal} if it
commutes with its Hilbert space adjoint, that is if
   \begin{equation} \label{1.19.11}
   NN\gw=N\gw N.
   \end{equation}
This purely algebraic definition can be made spatial through a
standard argument: $N$ is normal if and only if
   \begin{equation} \label{1.12.11}
   \|Nf\|=\|N\gw f\|,\quad f\in\hhc.
   \end{equation}
   Let us notify the following.
   \begin{try} \tlabel{t1.27.12}
   $N$ is normal if and only if \eqref{1.12.11} holds for $f$'s
from a dense linear space\,\footnote{\;If we want to have a linear
space closed we always make it clear.} $\ddc$ only.
   \end{try}

   \subsubsection*{Spectral
representation} The most powerful tool for normal operators is its
spectral representation. Though different people may have
different understanding of it, everyone agrees that the most
appealing is its spatial version below.
   \begin{thm}[Spectral Theorem] \label{specth}An operator $N$
is normal if and only if it is a spectral integral of the identity
function on $\ccb$ with respect to a spectral measure $E$ on
$\ccb$. Such a spectral measure $E$ is uniquely determined and its
closed support coincides with the spectrum of $N$.
   \end{thm}
%

   \subsubsection*{From spectral representation to $\llc^2$--model}

   What is sometimes meant by spectral theorem, tailored to the
simplest possible situation and as such pretty often satisfactory
in use, is the following.
   \begin{cor} \tlabel{t2.19.11}
  If an operator $N$ is normal and $*$--cyclic, then there is
positive measure $\mu$ on $\spek N$ such that $N$ is unitarily
equivalent to the operator $M_Z$ of multiplication by the
independent variable on $\llc^2(\mu)$.
   \end{cor}
   $N$ is {\em $*$--cyclic} means\,\footnote{\label{futa1}\;This
definition makes sense for any operator as long as the involved
monomials are kept to be ordered like $N\gw{}^kN^l$.}
 \underbar{here} that there is a vector $e\in\hhc$ (called a
$*$--cyclic vector of $N$) such that the linear space
   \begin{equation} \label{5.19.11}
   \colorbox{Light}{$\zb{p(N\gw,N)e}{p\in\ccb[Z,\sbar Z]}$}
   \end{equation}
   is \colorbox{Light}{dense} in $\hhc$; this notion appears as
one of the very sensitive when passing to unbounded operators.

   The converse to Corollary \ref{t2.19.11} is trivial. We state
it here because of the further role it is going to play.
   \begin{fact} \label{t3.19.11}
   Suppose $\mu$ is compactly supported positive
measure\,\footnote{\;We call a measure positive if it takes
non-negative values.} on $\ccb$. Then the operator $M_Z$ of
multiplication by the independent variable is normal and
$*$-cyclic with the cyclic vector $e=1$.
   \end{fact}

   Some supplementary information is in what follows.
   \begin{cor} \label{t2.11.24}
   Let $N$ and $\mu$ be as in {\rm Corollary \ref{t2.19.11}}. Then
   \begin{equation*}
   \is{N^me}{N^ne}=\int_\ccb z^m\sbar z^n\mu(\D z),\quad
m,n=0,1,\dots
   \end{equation*}

   \end{cor}

   \subsubsection*{The spectrum}  For the spectrum of a normal
operator $N$ we have
   \begin{equation*}
   \spek{N}={\rm sp_{ap}}(N)
   \end{equation*}
   where the right hand side stands for the approximate point
spectrum. As the adjoint $N\gw$ is normal as well the same refers
to it; the apparent equality $\spek N=\overline{\spek{N\gw}}$ is
applicable here.

   \subsection*{The finest points of subnormal operators}
   Here we would like to itemize the topics, which are well known
in the theory of bounded subnormal operators (cf. \cite{con}), and
which we are going to juxtapose with those for unbounded
operators.
   \subsubsection*{Normal dilations and subnormality}
   Given $A\in\bbs(\hhc)$, a normal operator $N\in\bbs(\kkc)$,
$\kkc$ contains isometrically $\hhc$, is said to be a ({\em
power}) {\em dilation} of $A$ if
   \begin{equation}\label{1.28.12}
   A^nf=PN^nf, \quad f\in\hhc,\quad n=0,1,\ldots
   \end{equation}
   with $P$ being the orthogonal projection of $\kkc$ onto $\hhc$;
if $N$ is a dilation of $A$ then so is $N\gw$ for $A\gw$.

   If for $S\in\bbs(\hhc)$ there is $N$ normal in $\kkc$ such that
instead of \eqref{1.28.12} we have
   \begin{equation} \label{2.28.12}
   Sf=Nf, \quad f\in\hhc,
   \end{equation}
   then we say that $S$ is {\em subnormal}. If $S$ is subnormal
and $N$ is its normal extension then $N\gw$ is a normal dilation
of $S\gw$. In addition to this we have, cf. \cite[\S 5]{nagy}
   \begin{pro}\tlabel{t1.30.12}
   The following conditions are equivalent:
   \begin{enumerate}
   \item[(a)] $B$ is an extension of $A$; \item[(b)] $B$ is a
dilation of $A$ and $B\gw B$ is a dilation of $A\gw A$; \item[(c)]
 $B\gw{}^i B^j$ is a dilation of $A\gw{}^i A^j$ for any
$i,j=0,1,\dots$.
   \end{enumerate}
   \end{pro}

   Another way of writing \eqref{2.28.12}, both illustrative and
precise, is
   \begin{equation} \label{3.28.12}
   S\subset N;
   \end{equation}
   use of $\subset$ suggests the graph connotation.

   \subsubsection*{Halmos' positive definiteness and Bram's characterization
of subnormality}
   It is an immediate consequence of normality of $N$ in
\eqref{2.28.12} that a subnormal operator $S\in\bbs(\hhc)$ must
necessarily satisfy a kind of {\em positive definiteness}
condition introduced by Halmos in \cite{halmos}:
   \begin{equation*}
   \sum_{m,n}\is{S^mf_n}{S^nf_m}\Ge 0, \quad \text{for any finite
sequence $(f_k)_k\subset\hhc$}. 
   \end{equation*}

   \begin{thm} \tlabel{t2.28.12}
   $S\in\bbs(\hhc)$ is subnormal if and only if it satisfies the
 positive definiteness condition and
   \begin{equation} \label{1.3.1}
\sum_{m,n}\is{S^{m+1}f_n}{S^{n+1}f_m} \Le
C\sum_{m,n}\is{S^mf_n}{S^nf_m} , \quad \text{for any finite
sequence $(f_k)_k\subset\hhc$}
   \end{equation}
   with some $C\Ge0$.
   \end{thm}
   Bram's result\,\footnote{\;A short replacement for Bram's main
argument concerning redundancy of \eqref{1.3.1} can be found in
\cite{pams}. The argument from \cite{pams} is present in \cite[p.
509]{nagy_rus}} \cite{bra} says the boundedness condition
\eqref{1.3.1} in Halmos' Theorem \ref{t2.28.12}\,\footnote{\;For
another proof of Halmos' theorem look at \cite{nagy}} is
 superfluous. It turns out that the boundedness condition
\eqref{1.3.1} comes back in the unbounded case under some forms of
growth conditions.

   Another characterization is in \cite{ando}; it is interesting
because it provides with a matricial construction of the extension
space independent of $S$. In principal it does lead to minimal
extensions, cf. Proposition \ref{t4.28.12}.

   \subsubsection*{Minimality and uniqueness of extensions}
   For $S$ subnormal and its normal extension $N$ let us take into
consideration the following three situations.
   \begin{enumerate}
   \item[(M${}_1$)] \label{min} If $\hhc\subset\kkc_1\subset\kkc$
and $N\rres{\kkc_1}$ turns out to be normal then either
$\kkc_1=\hhc$ or $\kkc_1=\kkc$.
   \end{enumerate}
   For $E$ being the spectral measure of $N$ and $\ddc$ a linear
subspace of $\hhc$ set
   \begin{align} \label{4.28.12}
   \ssc_\ddc&\okr\clolin\zb{E(\sigma )f}{\text{$\sigma$ Borel
subset of $\ccb$, $f\in\ddc$}}
   \\ \label{5.28.12}
   \ccc_\ddc&\okr\lin\zb{N\gw{}^mN^nf}{f\in\ddc,\,m,n=0,1,\ldots}.
   \end{align}
   \begin{enumerate}
   \item[(M${}_2$)] $\ssc_\hhc$ is $\kkc$. \item[(M${}_3$)] The
closure of $\ccc_\hhc$ is $\kkc$.
   \end{enumerate}
   The standard fact of the theory says the conditions (M${}_1$),
(M${}_2$) and (M${}_3$) are equivalent. If this happens we speak
of {\em minimality} of $N$. Notice minimal normal extensions
always exist, both (M${}_2$) and (M${}_3$) provide with an
algorithm to determine them. Moreover,
   \begin{pro}\tlabel{t4.28.12}
   Two minimal normal extensions of a subnormal $S$ are
$\hhc$--equivalent, that is there is a unitary similarity between
them which remains identity on the space $\hhc$.
   \end{pro}

   \section*{The hazardous terrain of unbounded operators}
   All the operators from now on are \underbar{densely}
\underbar{defined}; if $A$ is such, $\dz A$ always stands for its
domain. The closure of $A$, if it is closable, is denoted by
$\sbar A$. If $\ddc$ is a linear subspace of $\dz A$ then
$A\res\ddc$ stands \label{abc} for the {\em restriction} of $A$ to
$\ddc$.
   \subsection*{Normal operators and their spectral
representation again}

   The definition of normality in unbounded case is much the same,
more precisely, a \underbar{closed} operator is said to be {\em
normal} if \eqref{1.19.11} holds\,\footnote{\;It is always tacitly
understood that the domain of a composition of two operators is
the maximal possible one. One has to notice that the adventure
with domains of unbounded operators already starts here.}.
However, it turns out that a version of \eqref{1.12.11} is more
easy-to-use: $N$ is normal if and only if
   \begin{gather}
\label{2.19.11} \begin{split}
   \dz N=\dz{N\gw}, \phantom{aaaaa}
   \\
    \|Nf\|=\|N\gw f\|,\quad f\in\dz N. \end{split}
   \end{gather}
    Now closeness of $N$ is implicit in \eqref{2.19.11}.
   \subsubsection*{The plain version of spectral theorem} As in the bounded case all the versions of
spectral representation are available. The spectral theorem,
Theorem \ref{specth}, is true as stated due to the vast
flexibility of the spectral integral. We are going to state it
here with more particulars enhancing some of them which are
pertinent to unbounded operators; of course, they are present in
the bounded case as well.
   \begin{thm}[Spectral Theorem, the extras included] \tlabel{t1.25.11}
   An operator $N$ is normal if and only if it is a spectral
integral of the identity function on $\ccb$ with respect to a
spectral measure $E$ on $\ccb$, that is
   \begin{enumerate}
   \item[\liczp 1] $\is{Nf}{g}=\int_{ \ccb} z \is{E(\D z)f}{g}$
for all $f\in\dz N$ and $g\in\hhc$.
   \end{enumerate}
Moreover, if this happens then
   \begin{enumerate}
   \item[\liczp 2] $\dz
N=\zb{f\in\hhc}{\int_{ \ccb}|z|^2\is{E(\D z)f}{f}<+\infty}$;
   \item[\liczp 3]  for every  Borel measurable non-negative
function $\phi$ on $ \ccb$ and $f\in\dz N$
   \begin{equation*}
   \int_{ \ccb} \phi(x)\is{E(\D z)Nf}{Nf}=\int_ \ccb\phi(x)
\,|z|^2 \is{E(\D z)f}f,
   \end{equation*}
   in particular,
   \begin{equation} \label{1.9.5.7}
   \|Nf\|^2=\int_{ X}|z|^2\is{E(\D z)f}{f},\quad f\in\dz N
   \end{equation}
  and
   \begin{equation}\label{4.26.11}
   \text{$E(\sigma)N\subset NE(\sigma)$ for all Borel sets
$\sigma$};
   \end{equation}
   \item[\liczp 4]
   the spectral measure $E$ is uniquely determined and its closed
support coincides with the spectrum of $N$.
   \end{enumerate}

   \end{thm}
   This is more or less all what survives from surroundings of the
spectral theorem when passing from the bounded case to the
unbounded one.

   \subsubsection*{Invariant and reducing subspaces}
    A \underbar{closed} subspace $\mathcal L$ of $\hhc$ is {\it
invariant} for $A$ if $A(\mathcal L\cap\dz A) \subset\llc$; then
the {\em restriction} $A\rres{\mathcal L}\okr A\res{\mathcal
L\cap\dz A}$ is a operator in $\llc$ becomes clear. If $A$ is a
closable (closed) operator in $\hhc$ then so is $A\rres\llc$; this
is so because the notions are topological, with topology in the
graph space. The closed subspace $\llc$ is invariant for $A$ if
and only if $PAP=AP$, where $P$ is the orthogonal projection of
$\hhc$ onto $\mathcal L$.

    On the other hand, a \underbar{linear} subspace
$\ddc\subset\dz A$ is said to be {\it invariant} for an operator
$A$ in $\hhc$ if $A\ddc\subset\ddc$. If this happens and $\ddc$ is
not dense in $\hhc$ we consider the {\em restriction} $A\res\ddc$
as a densely defined operator in $\overline{\ddc}$. However, if
$\ddc$ is a dense in $\hhc$ then $A\res\ddc$ is still a densely
defined operator in $\hhc$.

   The above two concepts of invariance and restriction look much
alike. If a linear subspace $\ddc$ is invariant for $A$ then
$A\res\ddc\subset A\rres{\overline{\ddc}}$ whereas
$\overline{A\res\ddc}= \sbar A\rres{\overline{\ddc}}$ provided $A$
is closable\,\footnote{\; Identifying operators with their graphs
we can write $A\res\ddc=A\cap(\ddc\times\ddc)$ and
$A\rres\ddc=A\cap(\overline{D}\times\overline{D})$. Hence the
equality follows.}. This makes the difference more transparent.
\label{t}

 A step further, a closed subspace ${\mathcal L}$ {\it reduces} an
operator $A$ if both $\mathcal L$ and $\mathcal L^\perp$ are
invariant for $A$ as well as $P\dz A\subset\dz A$; all this is the
same as to require $P A\subset AP$. The restriction
$A\rres{\mathcal L}$ is called a {\it part} of $A$ in $\mathcal
L$.

   \subsubsection*{$\ccc^\infty$--vectors} For  an operator $A$
set
   \begin{gather*}
   \ddc^\infty(A)\okr\bigcap_{n=0}^\infty\dz{A^n},
   \\
\ddc^\infty(A,A\gw)\okr\bigcap_{\substack{{A_1,\ldots A_n\in\{A\gw,A\}}\\
{\text{any finite choice}}}} \dz{A_1\cdots A_n}.
   \end{gather*}
   It is customary to refer to vectors in any of these two classes
as to $\ccc^\infty$--ones.

One has to notify that
   \begin{equation*}
   \ddc^\infty(A\gw,A)=\ddc^\infty(A,
A\gw)\subset\ddc^\infty(A\gw)\cap\ddc^\infty(A).
   \end{equation*}
   If $f\in\ddc^\infty(A)$ then $p(A)f\in\ddc^\infty(A)$ for any
$p\in\ccb[Z]$ as well, if $f\in\ddc^\infty(A\gw,A)$ then
$p(A\gw,A)f\in\ddc^\infty(A\gw,A)$ for any $p\in\ccb[Z,\sbar Z]$;
the latter regardless any commutativity property between $A$ and
$A\gw$, cf. footnote \ref{futa1}.

   A vector $f\in\ddc^\infty(A)$ may belong to one of the
following classes: $\bbc(A)$ ({\em bounded}), $\aac(A)$ ({\em
analytic}) or $\qqc(A)$ ({\em quasianalytic}). While the last two
are rather pretty well known we give here the definition of {\em
bounded vectors}, they are those $f$'s in $\ddc^\infty(A)$ for
which there are $a,b$ such that $\|A^nf\|\Le a b^n$, $n=0,1,\dots$
It is clear that
   \begin{equation*}
   \bbc(A)\subset\aac(A)\subset\qqc(A).
   \end{equation*}
   The first two linear subspaces whereas the third is
not\,\footnote{\;There are two more notions: seminanalytic and
Stieltjes vectors, they are rather less popular, cf.
\cite{book_sub}}.

   \subsubsection*{A core} This is an important invention for
   unbounded operators when a need not to consider them closed
becomes strong. Let us call here that this appear more often than
someone may imagine, take an operator with invariant domain, if it
is \underbar{closed}, then in the vast majority of cases it turns
out to be necessarily \underbar{bounded}, see \cite{ota1}. If
someone does deal with a closed operator and in spite of this
wants to consider an invariant domain a core comes to rescue. Thus
$\ddc\subset\dz A$ is a {\em core} of a closable\,\footnote{\;A
core may be defined even for non-closable operators because in
fact the graph topology is behind the notion.} operator $A$ if
$\overline{A\res\ddc}=\sbar A$. Trivially, a domain $\dz A$ is
always a core of $A$ and, on the other hand, a core must
necessarily be dense. The essence of the notion of
 core is in offering additional `domains' for an operator. On this
occasion we recall a practical notion: a closable $N$ is called
{\em essentially normal} if $\sbar N$ is normal.

   A handy necessary and sufficient condition for $\ddc$ to be a
core of $A$ is the following implication to hold
   \begin{equation} \label{3.26.11}
   \text{for $f\in\dz A$ such that $\is fg+\is{Af}{Ag}=0$ for all
$g\in\ddc$ implies $f=0$.}
   \end{equation}

   The observation which follows fits within the character of this
section and makes intrinsic use of the notion of core.
   \begin{pro}\tlabel{t2.26.11}
   Bounded vectors of a normal operator form a core of it.
Therefore, a normal operator decomposes as an orthogonal sum of a
sequence of bounded normal operators.
   \end{pro}
   \begin{proof}
    Due to \liczp 2 in Spectral Theorem \ref{t1.25.11} for any
bounded set $\sigma\subset\ccb$ and $f\in$ the vector $E(\sigma)f$
is in $\dz N$ and, by \eqref{4.26.11}, $
E(\sigma)NE(\sigma)f=NE(\sigma)f$. This means that the linear
space $\bbc(\sigma)\okr\zb{E(\sigma)f}{f\in\hhc}$ is invariant for
$N$. Because $\cup_\sigma\bbc(\sigma)$ is dense every $N^n$ is
normal as well. Therefore, by \eqref{1.9.5.7}, for any $f\in\hhc$
and any bounded set $\sigma$ $E(\sigma)f$ is a bounded vector. To
check that they all together constitute a core proceed as follows.
Due to \eqref{4.26.11}, $E(\sigma)N\gw Nf=N\gw E(\sigma)Nf=N\gw
NE(\sigma)f$ and therefore condition \eqref{3.26.11} gives
   \begin{equation*}
   E(\sigma)f+E(\sigma)N\gw Nf=0
   \end{equation*}
   Because $\sigma$ is an arbitrary bounded Borel set we infer
that $f+N\gw Nf=0$, hence $f=0$.

  Decomposing $\ccb$ as a disjoint sum of bounded Borel sets we
get the orthogonal decomposition in question. More precisely, if
$\{\sigma_n\}_n$ is such a partition of $\ccb$ then the subspaces
$E(\sigma_n)\hhc$ are mutually orthogonal and reduce $N$; this is
due to \eqref{4.26.11}. Notice that because the parts
$N\rres{E(\sigma_n)N}$ are bounded a graph argument guarantees the
orthogonal sum of the parts is a closed operator. Now because
bounded vectors form a core of $N$ the final conclusion comes out.
   \end{proof}
   \begin{cor} \tlabel{t3.3.1}
   Any $\ddc\in\{\bbc(N),\aac(N),\lin\qqc(N)\}$ is a core of a
normal operator $N$.
   \end{cor}

      \subsubsection*{Resemblance of normality: formal normality}
      The first and very serious surprise comes when one asks what
happens now to Triviality \ref{t1.27.12}. In the unbounded case
one gets nothing but $\ddc\subset\dz{(N\res\ddc)\gw}$. If $\ddc$
is a core of $N$ then \eqref{1.12.11} can be stated as
   \begin{gather}
   \begin{split}\label{1.27.12} \dz N\subset\dz{N\gw},\phantom{aaaaa}
   \\
    \|Nf\|=\|N\gw f\|,\quad f\in\dz N. \end{split}
   \end{gather}
and nothing more. Therefore, we have to call those $N$'s somehow.
Because \eqref{1.27.12} and \eqref{1.12.11} look much alike, the
name in use for operators satisfying \eqref{1.27.12} is: {\em
formally normal}. Though there is a tiny difference in definitions
of normality and formal normality, `$=$' is replaced by
`$\subset$', the consequences are rather significant as we are
going to realize later.

   Notice that if $N$ is formally normal then it must necessarily
be closable. Moreover, its closure $\sbar N$ is formally normal as
   \begin{equation} \label{1.18.1}
   \dz{\sbar N}\subset\dz{N\gw}.
   \end{equation}
 Moreover, if $N$ is formally normal and $\dz N$ is a core of
$N\gw$ then $N$ is essentially normal.

   \begin{pro}\tlabel{t1.1.1}
   Suppose $N$ is formally normal in $\hhc$. If $N_1$ is a normal
operator $N_1$ in $\hhc$ such that $N_1\subset N$ then $N$ is
normal too and $N_1=N$.
   \end{pro}
   \begin{proof}
   Because $N_1\subset N$, $N\gw\subset N\gw_1$ and consequently
$$\dz{N\gw}\subset\dz{N\gw_1}=\dz{N_1}\subset\dz N,$$ which makes
$N$ normal and equal to $N_1$.
   \end{proof}

      \subsubsection*{The operator of multiplication by
independent variable}
   Let $\mu$ be a positive measure on $\ccb$ of finite
moments\,\footnote{\label{tt}\;We say $\mu$ has finite moments if
$\int_\ccb|z|^{2n}\mu(\D z)<\infty$ for all $n=0,1,\dots$. This is
what we are taking for granted in this paper once and for all.}.
Denote by $\ppc(\mu)$ the polynomials in $\ccb[Z,\sbar Z]$
regarded as members of $\llc^2(\mu)$. Define the operator $M_Z$ of
multiplication by the independent variable in $\llc^2(\mu)$ as
   \begin{gather*}
   \dz {M_Z}\okr\zb{f\in\llc^2(\mu)}
{\int\nolimits_\ccb|zf(z)|^2\mu(\D z)<\infty},
   \\
   (M_Z f)(z)\okr zf(z),\; z\in\supp\mu,\quad \funk{M_Z} f{M_Z f}.
   \end{gather*}
   Notice that the characteristic (indicator) functions $1_\sigma$
of Borel subsets of $\ccb$ are in $\dz {M_Z}$. Therefore $M_Z$ is
densely defined, and because $(M_Z)\gw=M_{\sbar Z}$ as well as
$(M_{\sbar Z})\gw=M_{Z}$, the operator $M_Z$ is closed and,
consequently, it is normal.

   Suppose $\ppc(\mu)$ is dense $\llc^2(\mu)$. Then
$M_Z\res{\ppc(\mu)}$ is a densely defined operator. Is it
essentially normal? In general not because
   \begin{fact} \tlabel{t2.27.12}
   $M_Z\res{\ppc(\mu)}$ is essentially normal if and only if
$\ppc(\mu)$ is a core of $M_Z$. This happens if and only if
$\ppc(\mu)$ is dense in
$\llc^2((1+|Z|^2)\mu)$\,\footnote{\label{ultra}\;Such measures are
called in \cite{fu} {\em ultradeterminate}. By the way, a measure
is ultradeterminate if the polynomials in $\ppc(\mu)$ are dense in
some $\llc^p(\mu)$, $p>2$ (see \cite{fu}, p. 61).}.
   \end{fact}
   The second conclusion of the above follows immediately from the
fact that the space $\llc^2((1+|Z|^2)\mu)$ bears the graph norm
with respect to the operator $M_Z$.

   \begin{rem} \tlabel{t1.28.12}
The operator $M_Z\res{\ppc(\mu)}$ is formally normal in the
closure $\overline{\ppc(\mu)}$ of $\ppc(\mu)$ in
$\llc^2(\mu)$--norm regardless $\ppc(\mu)$ is dense in
$\llc^2(\mu)$ or not and has a normal extension $M_Z$ in
$\llc^2(\mu)$. In other words, $M_Z\res{\ppc(\mu)}$ is always
formally normal, has a normal extension in $\llc^2(\mu)$ though it
may act within a \underbar{smaller} space $\overline{\ppc(\mu)}$.
   \end{rem}

      \subsubsection*{Repairing $*$--cyclicity}
      The notion $*$--cyclicity, as defined in the greyish area
around \eqref{5.19.11} for bounded operators, for unbounded ones
requires \eqref{5.19.11} to hold for $f\in\ddc^\infty(N,N\gw)$.
The above considerations show that this definition is not
satisfactory in the unbounded case for quite a number of reasons:
neither Corollary \ref{t2.19.11} nor Fact \ref{t3.19.11} holds
true in particular.
    Therefore, call now $N$ $*$-{\em cyclic}\,\footnote{\;It is
tempting to call it rather {\em graph} $*$--{\em cyclic} as graph
topology is behind this. Regrettably, we have to abandon this
appeal; also because present term includes trivially that for
bounded operators.} with a cyclic vector $e\in\ddc^\infty(N\gw,N)$
if the set\,\footnote{\;The remark made in footnote \ref{futa1}
applies here as well.} \eqref{5.19.11} is a core of $N$. Under
this modification both Corollary \ref{t2.19.11} and Fact
\ref{t3.19.11} \underbar{revive}.

    \subsubsection*{A word about spectral properties}
    An example of an ultradeterminate measure is the Gaussian one,
that is $\E^{-|x|^2}\D x$. The polynomials in $\ppc(\mu)$
constitute a core of $M_Z$ and all the oddities are left apart.
However, here $\spek N=\ccb$ which \underbar{excludes} any
resolvent tool to be used; this is what someone ought to take into
account when trying to approach the theory.

    \subsection*{Assorted topics on unbounded subnormals}
    \subsubsection*{Subnormality and its characterization}
     The defining formula \eqref{3.28.12} remains working also in
the unbounded case; more precisely an operator $S$ densely defined
in a Hilbert space is called {\em subnormal} if there exists a
normal operator $N$ is a Hilbert space $\kkc$ containing
isometrically $\hhc$ such that \eqref{3.28.12} holds true. Another
way of expressing this is that $\hhc$ is invariant for $N$ and
$S\subset N\rres\hhc$.

   The only characterization of subnormality which does not impose
any constrain on behaviour of domains of the operator is that via
semispectral measures\,\footnote{\;A semispectral measure differs
from a spectral one by dropping the assumption its values are
orthogonal projections; it is also known under the name `positive
operator valued measure'.} (see, \cite{bishop} or \cite{fojasz})
or its versions (like in \cite{sesq} and \cite{cyk}).
    \begin{thm}\label{2x}
    An operator $S$ is subnormal if and only if there is a
    semispectral measure $F$ on Borel sets of $\ccb$ such
    that\,\footnote{\;Condition \eqref{6.28.12} corresponds to
    those in Proposition \ref{t1.30.12}.}${}^{,}$\,\footnote{\; If
    \eqref{6.28.12} holds only for $m=1$ and $n=1$ ($m=n=0$ is a
    triviality) then $S$ has a normal dilation exclusively and
    {\em vice versa}. In that case the fourth condition encoded in
    \eqref{6.28.12} downgrades to the inequality
    $\is{Sf}{Sg}\Le\int_{\mathbb C}|z|^2\is{F(\D z)f}{g}$,
    $f,g\in\dz S$.}
    \begin{equation}\label{6.28.12}
    \is{S^mf}{S^nf}=\int_{\mathbb C}z^m\sbar z^n\is{F(\D
    z)f}{g},\quad m,n=0,1, \quad f,g\in\dz S.
    \end{equation}
    \end{thm}
   Notice that semispectral measures related to a subnormal
operator may not be uniquely determined, see \cite{expl} for an
explicit example. As spectral measures of normal extensions come
via dilating semispectral measure, according Naimark's dilation
theorem, cf. \cite{nai}, we may have quit a number of them as
well. This foretells somehow the problem with uniqueness (and
minimality) we are going to expose a little bit later. So far we
turn Theorem \ref{2x} into an equivalent form involving scalar
spectral measures, cf. \cite{cyk}.

   Call a family $\{\mu_{f}\}_{f\in\hhc}$ of positive measures on
$\ccb$, a family of {\em elementary spectral measures} of $S$ such
that for $f,g\in\hhc$
    \begin{gather} \label{3.17.1}
   \text{$\mu_{\lambda f}(\sigma)=|\lambda|^2\mu_f(\sigma)$ for
$\lambda\in \mathbb C$,}\quad \mu_f(X)=\|f\|^2
    \\ \label{4.17.1}
    \mu_{f+g}(\sigma)+\mu_{f-g}(\sigma)=2(\mu_f(\sigma)
    +\mu_g(\sigma))
    \end{gather}
   and
    \begin{equation*}
    \is{S^mf}{S^nf}=\int_{\mathbb C}z^m\sbar z^n\mu_{f}(\D
    z),\quad m,n=0,1, \quad f\in\dz S.
    \end{equation*}
   \begin{thm}[A version of Theorem \ref{2x}] \tlabel{t2.30.12}
   An operator $S$ is subnormal if and only if it has a family of
elementary spectral measures.
 \end{thm}
    As an immediate consequence of Theorem \ref{2x} we get a
slight extension (no domain invariance required) of Proposition 18
in \cite{try4}. Though this very much wanted observation looks
trivially no direct way of getting it from the definition of
subnormality seems to be available. This is so because, unlike
normality, the definition of subnormality `exceeds the underlying
Hilbert space'. Here a kind of exception is Ando's construction of
the universal extension space in which the unitary equivalence can
be placed in. However, in the unbounded case this construction
does not look it to work at the full, cf. \cite{try3}.

   \begin{cor} \tlabel{t1.31.12}
   Let $S$ be an operator $\hhc$ let $\funk V \hhc{\hhc_1}$ be a
   bounded operator such that $V\gw VS=S$. If $S$ is subnormal in
   $\hhc$, then so is $VSV\gw$ in $\hhc_1$ provided it is densely
   defined. More exactly, if $F$ is a semispectral measure of $S$
   then $VF(\,\cdot\,)V\gw$ is such for $VSV\gw$.
   \end{cor}

   \subsubsection*{Minimality and uniqueness} Minimality in the
unbounded case becomes a very sensitive issue. Let us start with a
definition: call $N$ {\em minimal of spectral type} if (M${}_1$)
on page \pageref{min} is satisfied. It turns out that, cf.
Proposition 1 in \cite{try3}, it is equivalent to (M${}_2$) in the
sense that $\ssc_{\dz S}=\kkc$. The sad news is that minimal
normal extensions of spectral type may not be
$\hhc$--equivalent\,\footnote{\;It is right time to give the
definition: two extensions $B_1$ and $B_2$ in spaces $\kkc_1$ and
$\kkc_2$ of $A$ in $\hhc$ are called $\hhc$--{\em equivalent} if
there is a unitary operator $\funk{U}{\kkc_1}{\kkc_2}$ such that
$U\kkc_1=\kkc_2U$ and $U\rres{\hhc}=I_\hhc$.}, see Example 1 in
\cite{try3} much further developed in \cite{deter}; therefore
\underbar{no} uniqueness can expected at this stage. The good news
is the welcomed spectral inclusion
   \begin{equation}\label{7.2.1}
  \spek{N}\subset\spek{S}
   \end{equation}
   is preserved; as a consequence of \eqref{7.2.1} notify $\spek
S\neq\nic$. A list of further spectral properties is in Theorem 1
of \cite{try3}.

The third kind of minimality appearing in (M${}_3$) though well
defined cannot be well developed in this general setting. It does
when $S$ has an invariant domain; we come to this latter on.

   \subsubsection*{Tightness of normal extensions}
   Assume for a little while $A$ in $\hhc$ and $B$ in $\kkc$ are
arbitrary operators, $\kkc$ contains isometrically $\hhc$. If
$A\subset B$ then
   \begin{equation} \label{2.2.1}
   \text{$\dz A\subset\dz B\cap\hhc$ and
$P\dz{B\gw}\subset\dz{A\gw}$ with $A\gw Px=PB\gw x$ for
$x\in\dz{B\gw}$}
   \end{equation}
   with $P$ being apparently the orthogonal projection of $\kkc$
onto $\hhc$. If $N$ is formally normal extension of $S$ then both
inclusions in \eqref{2.2.1} merge in one
   \begin{equation} \label{3.2.1}
   \dz S\subset\dz N\cap\hhc\subset \dz{N\gw}\cap\hhc\subset
P\dz{N\gw}\subset\dz{S\gw}.
   \end{equation}
   This implies immediately that
   \begin{equation*}
   \text{$\dz {S}\subset \dz{S\gw}$ and $\|S\gw f\|\Le\|Sf\|$ for
$f\in\dz S$}.
   \end{equation*}
   Hence $S$ is closable and $\dz{\sbar S}\subset \dz{S\gw}$; the
latter has to be compared with \eqref{1.18.1}.

Call the extension $N$ {\em tight} if $\dz{\sbar S}=\dz N\cap\hhc$
and {\em $*$-tight} if $P\dz{N\gw}=\dz{S\gw}$, cf. \cite{ccr}, the
topic was taken up in \cite{henki}. Notice that tight
extendibility was one of the condition involved in the definition
of subnormal operators given in \cite{sun}. It was proved in
\cite{ass} that symmetric and analytic Toeplitz operator
 have tight extension. The question in \cite{ass} asks if this is
always the case, which would give subnormality of \cite{sun} the
same meaning as ours. It turns out they two different notions
according to the example in \cite{ota2}. Therefore the preference
is the present one.

   \subsubsection*{Cartesian decomposition}
   If $A$ has $\dz A\cap\dz{A\gw}$ dense then
   \begin{equation*}
   \RE A\okr \frac 12(A+A\gw),\quad \IM A\okr \frac
   1{2\I}(A-A\gw),\quad \dz{\RE A}=\dz{\IM A}=\dz A\cap\dz{A\gw}
   \end{equation*}
   leads to the Cartesian decomposition of $A$
   \begin{equation*}
   A=\RE A+\I\IM A
   \end{equation*}
   with $\RE A$ and $\IM A$ symmetric on $\dz A\cap\dz{A\gw}$.
   \begin{pro}\tlabel{t1.1.1}
   A formally normal operator $N$ is essentially normal if and
only if the operators $\RE N$ and $\IM N$ are essentially
selfadjoint and spectrally commute, that is there spectral
measures commute. An operator $S$ is subnormal if an only if it
has such a formally normal extension.
   \end{pro}
   \begin{proof}
   If $N$ is formally normal and $\overline{\RE N}$ and
$\overline{\IM N}$ commute spectrally then $\overline{\RE
   N}+\I\overline{\IM N}$ is normal. Furthermore,
   $N\subset\overline{\RE N}+\I\overline{\IM N}\subset
   \overline{\RE N+\I\IM N}=\sbar N$ and Proposition \ref{t1.1.1}
   concludes with $N$ to be essentially normal. The rest follows
   easily.
   \end{proof}
   This is a parallel to Theorem \ref{2x}. It show that famous
Nelson's example from \cite{nel} can be adopted as an alternative
one to Coddington's \cite{cod}.

   \subsubsection*{Polar decomposition and quasinormal operators}
    For a closed densely defined operator $A: \hhc \supset \dz A
\longrightarrow \kkc$ there exists a unique partial isometry $U
\in \bbs (\hhc,\kkc)$ such that $\jd U = \jd A$ and $A=U|A|$,
where {$|A| \okr (A^* A)^{1/2}$}; $U|A|$ is called the {\it {polar
decomposition}} of $A$ (cf.\ \cite[p.\ 197]{wei} ). If $U|A|$ is
the polar decomposition of $A$, then $\overline{\ob{|A|}} =
\overline{\ob{A^*}}$ is the initial space of $U$ and
$\overline{\ob{A}}$ is the final space of $U$.

   One of the equivalent definitions of {\em quasinormal}
operators says they are those for which in their polar
decomposition $U|A|=|A|U$. These operators are
\underbar{subnormal} (cf. \cite[Theorem 2]{try2}), even more, they
have a kind of Wold-von Neumann decomposition, see \cite{brown}
for bounded operators and its version adapted to the unbounded
case \cite{book_sub}. In a sense they become an intermediate step
between subnormal and normal operators. Normal operators are those
quasinormals for which $\jd{A\gw}\subset\jd A$.

Because $\jd N=\jd{N\gw}$ for a normal $N$, both factor in its
polar decomposition can be extended properly so as to get the
following result.
   \begin{pro}\tlabel{t3.1.1}
   $N$ is normal if and only if $N=UP$ with $U$ unitary and $P$ a
positive operator, $U$ and $P$ commuting. This decomposition is
not unique.
   \end{pro}

    \subsubsection*{Old friends in the new environment} Because
selfadjoint operators are apparently normal, symmetric operators
are both formally normal and subnormal. The following draft shows
how all the notions interplay; all the inclusions
may\,\footnote{\;In the finite dimensional case subnormals,
formally normals and normals coincide.} become proper. Notice the
formally normal are somehow apart, formally normal operators may
not be normal, see \cite{cod} for an explicit example.
    \vspace{10pt} {\begin{center} $\begin{array}{ccccc}
    {\textit{selfadjoint}}&{\subset}&{ \textit{symmetric}}&
    {\subset} & {{\textit{formally normal}}}
   \\ {\cap} & {}&  {\cap}&{}&{}
   \\
   {\text{normal}} & \subset& {\textbf{subnormal}} & {} &{}
   \\  \cap &{}&{}&{}&{}
   \\
   {\textit{formally normal}}&{}&{}&{}&{}
   \end{array}$
    \end{center}}
    \vspace{10pt} Let us mention that Coddington characterizes in
\cite{cobir,codd} those formally normal operators which are
subnormal.

   \subsection*{Subnormality of operators with invariant domain}
   From now \underbar{onwards} we declare \vspace{2pt}
   \begin{center} \fbox{\;\,$S\dz S\subset\dz S.$}\end{center}
\vspace{5pt}
   This means we have to resign the temptation to consider an
operator $S$ to be closed unless we want deliberately exclude
operators which not bounded.

   Under these circumstances we have supplementing results to
   Theorems \ref{2x} and \ref{t2.30.12} at once.
   \begin{thm} \tlabel{t2xs}
   If $S$ is subnormal and $F$ is a semispectral measure such that
   \eqref{6.28.12} holds then \eqref{6.28.12} holds for all $m,n$,
   that is
   \begin{equation*}
    \is{S^mf}{S^nf}=\int_{\mathbb C}z^m\sbar z^n\is{F(\D
    z)f}{f},\quad m,n=0,1,\ldots \quad f\in\dz S.
    \end{equation*}
   Alternatively, the elementary spectral measures of $S$ satisfy
   \begin{equation}\label{2.17.1}
    \is{S^mf}{S^nf}=\int_{\mathbb C}z^m\sbar z^n\mu_{f}(\D
    z),\quad m,n=0,1,\ldots \quad f\in\dz S.
    \end{equation}
   \end{thm}
    \subsubsection*{Back to Halmos' positive definiteness
or what has survived from Bram's theorem}
   Under the current circumstances Halmos' positive definiteness
takes the form
   \begin{equation}
   \sum_{m,n}\is{S^mf_n}{S^nf_m}\Ge 0, \quad \text{for any finite
sequence $(f_k)_k\subset\dz S$}. \tag{{\tt PD}}
   \end{equation}
   What we still have in the flavour of Bram's characterization is
in the following, see \cite{try2} or \cite{nag_ext} for another
techniques of building the proof up.
   \begin{thm} \tlabel{t2.3.1}
   An operator $S$ in $\hhc$ satisfies {\rm(}{\tt PD}{\rm)} if and
only if there is a Hilbert space $\kkc$ containing $\hhc$
isometrically, and a formally normal operator $N$ in $\kkc$ such
that $S\subset N$ as well as
   \begin{equation}\label{2.7.1}
    \text{$N\dz N\subset \dz N$ and $N\gw\dz{ N}\subset \dz {N }$}
   \end{equation}
  If this happens, $N$ can be chosen to satisfy
   \begin{equation}\label{1.6.1}
   \dz N=\lin\zb{N\gw{}^kf}{k=0,1,\ldots,\;f\in\dz S}.
   \end{equation}

   \end{thm}
   \begin{rem} \tlabel{t10.10.1}
   Suppose $S$ and $N$ are as in {\rm Theorem \ref{t2.3.1}}. If
   $S$ is cyclic with a cyclic vector $e$ then $N$ is $*$--cyclic
with the same vector $e$. 
   Indeed, if $S$ is cyclic with a cyclic vector $e$ then, by
   \eqref{1.6.1},
$$\dz N=\lin\zb{N\gw{}^kN^le}{k,l=0,1,\ldots}$$ and the first
conclusion follows.
   \end{rem}

   %

   \begin{cor} \tlabel{t1.4.1}
   If $S$ is subnormal then it satisfies {\rm(}{\tt PD}{\rm)}.
   \end{cor}
   \begin{cor} \tlabel{t1x.8.1}
   $N$ in {\rm Theorem \ref{t2.3.1}} is essentially normal if and
only if
   \begin{equation*}
  x\in\dz {N\gw}\;\;\&\;\; \is xy+\is{x}{N\gw Ny}=0\;\; \forall
\;\;y\in\dz{N}\;\implies\;x=0.
   \end{equation*}
   \end{cor}
   \begin{proof}
   Notice first that $N$ is essentially normal if and only if $\dz
N$ is a core of $N\gw$. Now use \eqref{3.26.11} and \eqref{2.7.1}.
   \end{proof}

   We separate the uniqueness result because of its importance.
   \begin{thm} \tlabel{t2.6.1}
   If two pairs $(N_1,\kkc_1)$ and $(N_2,\kkc_2)$ satisfy the
conclusion of {\rm Theorem \ref{t2.3.1}} then they are
$\hhc$-equivalent, that is there is a unitary operator between
$\kkc_1$ and $\kkc_2$ such that $U\rres{\hhc}=I_\hhc$ and
$UN_1=N_2U$.
   \end{thm}
   \begin{proof}
   For $(f_k)_k\subset\dz S$ we have, due to \eqref{1.6.1},
   \begin{multline*}
\|\sum_nN_1\gw{}^nf_n\|_1^2=\sum_{k,l}\is{N_1^kf_l}{N_1^lf_k}_1=
\sum_{k,l}
\is{S^kf_l}{S^lf_k}\\=\sum_{k,l}\is{N_2^kf_l}{N_2^lf_k}_2=
\|\sum_nN_2\gw{}^nf_n\|_2^2
\end{multline*}
   which establishes the unitary operator between two dense
subspaces. The next step is standard as well.
   \end{proof}
   \begin{cor} \tlabel{t3.18.1}
   Suppose $S$ is subnormal in $\hhc$. If $\widetilde N$ is any
normal extension of $S$ and $ N$ is a formally normal extension of
$S$ as in {\rm Theorem \ref{t2.3.1}}, for which \eqref{1.6.1}
holds, then there is a formally normal operator $N_1$ which is
$\hhc$-equivalent to $N$ and such that $S\subset N_1\subset
\widetilde N$.
   \end{cor}
   \begin{proof}
  If $\widetilde N$ is normal in $\widetilde\kkc$ say, then the
subspace
   \begin{equation} \label{1}
   \ddc\okr \lin\zb{\widetilde N\gw{}^nf}{n=0,1,\dots,\;f\in\dz S}
   \end{equation}
of ${\dz{\widetilde N}}$ is invariant for $\widetilde N$ and
$\widetilde N\gw$. The operator $ N_1\okr\widetilde N\res{\ddc}$
is formally normal. Indeed, because, due to \eqref{2.2.1},
$\widetilde N\gw\res\ddc\subset(\widetilde
N\rres{\overline{\ddc}})\gw\subset (\widetilde N\res\ddc)\gw$ we
can write for $x\in \ddc$
   \begin{equation*}
   \|N_1x\|=\|\widetilde Nx\|=\|\widetilde N\gw x\|=\|(\widetilde
N\res\ddc)\gw x\| =\|N_1\gw x\|.
   \end{equation*}
Comparing \eqref{1} with \eqref{1.6.1} suggests the definition of
the required unitary $\hhc$-equivalence.
   \end{proof}

   \begin{cor} \tlabel{t3.7.1}
   If $S$ is a cyclic and subnormal operator then the formally
normal operator determined by\ {\rm Theorem \ref{t2.3.1}} can be
realized as the operator {$M_Z{}\res{\ppc(\mu)}$} in the
$\llc^2({\mu})$--closure of the polynomials $\ppc(\mu)$, where
$\mu\okr\is{F(\,\cdot\,)1}1$ and $F$ is an arbitrary semispectral
measure of $S$. According to {\rm Theorem \ref{t2.6.1}} any two
such models are $\hhc$--equivalent. Finally, $N$ itself is
subnormal\,\footnote{\; Look at Corollary \ref{t1.31.12}.}.
   \end{cor}

   In general, Theorem \ref{t2.3.1} is nothing but an intermediate
step toward subnormality. Because $N$ is just formally normal the
uncertainty is still ahead. The only known result which reminds
that of Bram is as follows, cf. \cite{try2}.
   \begin{thm} \tlabel{t3.3.1}
   Suppose $S$ is a weighted shift\,\footnote{\;$S$ in $\hhc$ is a
{\em weighted shift} if there is an orthonormal basis $(e_n)_n$ in
$\hhc$ such that $Se_n=\sigma_ne_{n+1}$ with some positive weights
$(\sigma_n)_n$.}. Then $S$ is subnormal if and only if it
satisfies the positive definiteness condition {\rm(}{\tt
PD}{\rm)}.
   \end{thm}
    \subsubsection*{Cyclicity and related matters}
     Getting experienced already with $*$--cyclicity we can say
that a closable operator $A$ with invariant domain is {\em cyclic}
with a {\em cyclic vector} $e$ if
   \begin{equation*}
   \zb{p(A)e}{p\in\ccb[Z]}
   \end{equation*}
   is a core of $A$. On the other hand, given a vector $f\in\dz A$
set
   \begin{equation*}
   \ddc_f\okr\zb{p(A)f}{p\in\ccb[Z]},\quad
\hhc_f\okr\overline{\ddc_f},\quad A_f\okr A\res{\ddc_f}.
   \end{equation*}
   The definition of $A_f$ is in accordance with what is on p.
\pageref and means an operator acting in the Hilbert space
$\hhc_f$; call $A_f$ the {\em cyclic portion} of $A$ {\em at} $f$.
Therefore $A$ is cyclic if and only if $\sbar A=\overline{A_f}$
for some $f\in\dz A$.

   Notice that if $g\in\ddc_f$ then $\ddc_g\subset\ddc_f$.
However, if $f\neq g$ we can not say anything reasonable about
dislocation of the spaces $\ddc_f$ and $\ddc_g$ unless they both
are reducing.

    \subsubsection*{The complex moment problem}

   Given a bisequence $(c_{m,n})_{m,n=0}^\infty$, call it a {\em
complex moment sequence} it there exists a positive Borel measure
$\mu$ on $\ccb$ such that
   \begin{equation}\label{1.4.1}
    c_{m,n}=\int_{\mathbb C}z^m\sbar z^n\D\mu,\quad m,n=0,1,\ldots
    \end{equation}
    The {\em complex moment problem} related to a bisequence
$(c_{m,n})_{m,n=0}^\infty$ consists in finding a measure {\em
representing} the bisequence via \eqref{1.4.1}\,\footnote{\;It
happens people carelessly mix up those concepts.}. The measure
$\mu$, thus the moment sequence $(c_{m,n})_{m,n=0}^\infty$, is
called {\em determinate} it there is no other measure representing
the sequence by \eqref{1.4.1}. Another, stronger concept,
introduced in \cite{fu}, calls the measure $\mu$, as well as the
related bisequence\,\footnote{\;The definition in \cite{fu} is
stated for a bisequence, that for a measure comes from searching
through the paper.}, {\em ultradeterminate}, cf. footnote
\ref{ultra}, if the operator $M_Z\res{\ppc(\mu)}$ of
multiplication by the independent variable is essentially normal.
If this happens, $M_Z$ is normal in $\llc^2(\mu)$, cf. Fact
\ref{t2.27.12}.

   The moment problem version of Corollary \ref{t1.4.1} determines
a kind of positive definiteness that a complex moment bisequence
$(c_{m,n})_{m,n=0}^\infty$ has necessarily to satisfy.
   \begin{pro}\tlabel{t2.4.1}
   If $(c_{m,n})_{m,n=0}^\infty$ is a complex moment bisequence
then
   \begin{equation}
\text{$\sum_{m,n} c_{m+q,n+p}\lambda_{m,n}\bar\lambda_{p,q}$ for
all finite bisequences $(\lambda_{k,l})_{k,l}\subset\ccb$}.
\tag{{\tt MPD}}
   \end{equation}

   \end{pro}
   In the other direction again we stop halfway.
   \begin{thm} \tlabel{t2a.3.1}
   A bisequence $(c_{m,n})_{m,n=0}^\infty$ satisfies {\rm(}{\tt
   MPD}{\rm)} if and only if there is a Hilbert space $\kkc$
   containing the $1$--dimensional Hilbert space $\ccb$
   isometrically, and a formally normal operator $N$ in $\kkc$
   such that $\dz N=\lin\zb{N\gw{}^mN^n1}{m,n=0,1,\ldots}$ and
   \begin{equation*}
   c_{m,n}=\is{N^m1}{N^n1}, \quad m,n=0,1,\ldots
   \end{equation*}
   \end{thm}
    Now everything depends on if the formally normal operator $N$
has a normal extension or not.

The explicitly defined bisequence in \cite{step}, p.259, which in
turn is an adapted to the present circumstances version of that in
\cite{fri}, satisfies ({\tt MPD}) and is not a complex moment one.

   \subsubsection*{Subnormality and the complex moment problem}
   Here we have the fundamental result which continues Theorem
\ref{t2a.3.1} and goes towards our open problem.
   \begin{thm} \tlabel{t2.4.1}
    {\rm(}a{\rm)} If $S$ is subnormal then
$(\is{S^mf}{S^nf})_{m,n=0}^\infty$ is a complex moment bisequence
for every $f\in\dz S$;

 {\rm(}b{\rm)} If $S$ is cyclic with a cyclic vector $e$ and
$(\is{S^me}{S^ne})_{m,n=0}^\infty$ is a complex moment bisequence
then $S$ is subnormal. Moreover, if $\mu$ is a measure which
represents $(\is{S^me}{S^ne})_{m,n=0}^\infty$ by \eqref{1.4.1} and
$f\in\dz S$ is of the form $p(S)e$ for some $p\in\ccb[Z]$ then
$(\is{S^mf}{S^nf})_{m,n=0}^\infty$ is a moment bisequence and
$|p|^2\mu$ is its representing measure.

   \end{thm}

   Referring to Theorem \ref{t2a.3.1} and Fuglede's classification
of determinacy we have two relevant notions: call a vector
$f\in\dz S$ a {\em vector of determinacy} of $S$ if the moment
bisequence $(\is{S^mf}{S^nf})_{m,n=0}^\infty$
 is determinate; if this bisequence is ultradeterminate call $f$
the {\em vector of ultradeterminacy} of
$S$\,\footnote{\;\label{ultra}The term `vector of uniqueness' as
in \cite{nuss}, which is more appropriate for symmetric operators
and real one dimensional moment problems, splits here in two.
Notice that in \cite{cyk} we use the term `vector of uniqueness
with still a slightly different meaning.}. It is clear these two
kinds of determinacy \underbar{require} the operator $S_f$ to be
already subnormal.

   \ass{{\rm In the discussion which follows there are two
alternating situations: they concern either a cyclic
\underbar{operator} or a cyclic \underbar{portion} of an operator.
A reader may chose to think of any of these two without any side
effect.}\vspace{8pt}} \rm

   \begin{thm} \tlabel{t3.4.1}
   Suppose $S$ is cyclic and $e$ is its cyclic vector. Then the
following two conclusions hold.
\begin{enumerate}
   \item[$(\alpha)$] If $e$ is the vector of determinacy then the
formally normal operator determined by\ {\rm Theorem \ref{t2.3.1}}
can be realized as the operator {$M_Z{}\res{\ppc(\mu)}$} in the
$\llc^2({\mu})$--closure of polynomials $\ppc(\mu)$, where $\mu$
is the unique measure representing
$(\is{S^me}{S^ne})_{m,n=0}^\infty $.

   \item[$(\beta)$] If $e$ is a vector of ultradeterminacy of $S$
then the formally normal operator $N$ constructed as in {\rm
Theorem
 \ref{t2.3.1}} is essentially normal. $\sbar N$ can be realized as
the normal operator $M_Z$ in $\llc^2({\mu})$ with $\mu$ as in
$(\alpha)$.
   \end{enumerate}
   \end{thm}
   \begin{proof}
   Apply Corollary \ref{t3.7.1} to get $(\alpha)$. Now use
$(\alpha)$ and the fact that $e$ is the vector of ultradeterminacy
of $S$ to come to $(\beta).$
    \end{proof}

   Notice ($\beta$) says that if $e$ is a vector of
ultradeterminacy for $S$ then it is so for the formally normal
operator $N$ constructed as in {\rm Theorem
 \ref{t2.3.1}}. In other words, the property of a vector to be
that of ultradeterminacy can be lifted to the extending space;
this is a rough comment rather then a precise statements.

   The next two results can be viewed as a \underbar{global}
version of Theorem \ref{t3.4.1}; the latter to be though of as a
\underbar{local} one.

 \begin{thm} \tlabel{t6.10.1}
   The two following two conclusions hold.
   \begin{enumerate}
   \item[$(\alpha')$] If every $f\in\dz S$ is a vector of
determinacy of $S$ and for the (unique) family $(\mu_f)_{f\in\dz
S}$ of measures representing the complex moment bisequence
$(\is{S^f}{S^nf})_{m,n}$, $f\in\dz S$, one has
   \begin{equation} \label{1.29.3}
   \mu_{f+g}+\mu_{f-g}-2\mu_g\Ge0,\quad f,g\in\dz S,
   \end{equation}
then $S$ is subnormal and has a \underbar{unique} normal extension
which is minimal of spectral type, and {\em conversely}.

\noindent Therefore, a formally normal operator $N$ can be
constructed as in {\rm Theorem
 \ref{t2.3.1}} and it is subnormal as well.
    \item[$(\beta')$] If the set $\uuc(S)$ of vectors of
ultradeterminacy of $S$ is total in $\hhc$ then the formally
normal operator $N$ constructed as in {\rm Theorem
 \ref{t2.3.1}} is essentially normal.
   \end{enumerate}

   \end{thm}

   \begin{proof}
  Proof of $(\alpha')$. It is clear there a unique family of
measures $\mu_f$, $f\in\dz S$ such that \eqref{2.17.1} and
\eqref{3.17.1} holds. The only condition missing so far to end up
with the conclusion is \eqref{4.17.1}.

   Take $f,g\in\dz S$ and with $m,n=0,1,\dots$ write
   \begin{multline*}
   \int_\ccb z^m\sbar z^n \D\big(\ulamek
12(\mu_{f+g}-\mu_{f-g})-\mu_f\big)=\ulamek 12\is{S^m(f+g)}{f+g}
+\ulamek 12\is{S^m(f-g)}{f-g}\\ -\is{S^mf}{f} =\is{S^mg}{g} =
\int_\ccb z^m\sbar z^n \D\mu_g.
   \end{multline*}
   Now \eqref{1.29.3} and determinacy at $g$ makes \eqref{4.17.1}
hold. Therefore, $S$ is subnormal due to Theorem \ref{t2.30.12}.
Corollary \ref{t3.18.1} establishes the final conclusion in
$(\alpha')$ concerning $N$.

   Proof of $(\beta')$. Let $N$ be the formally normal operator
constructed as Theorem \ref{t2.3.1}. Set
   \begin{equation}\label{2.29.3}
   \ddc_f(N)\okr\lin\zb{p(N\gw,N)f}{p\in\ccb[Z,\sbar Z]},\quad
N_f\okr N\rres{\overline{\ddc_e(N)}}
   \end{equation}
   and denote by $P_f$ the orthogonal projection on
$\overline{\ddc_e(N)}$. Notice that for $f\in\uuc(S)$ the operator
$\sbar N$ is normal. According to Lemma 2 of \cite{part} the
subspace $\hhc_f$ reduces $N$. Because $\uuc(S)$ is total,
   \begin{equation*}
\lin\zb{\ddc_f}{f\in\uuc(S)}=\dz N.
   \end{equation*}
   Adapting arguments used in the proof of Theorem of \cite{part}
we can check that $N$ is essentially normal.
   \end{proof}

   \subsubsection*{{{Minimality and uniqueness again}}}
   Now is a right time to come back the minimality problem of
extensions of $S$. An extension $N$ of $S$ {\em minimal of cyclic
type} if $\ccc_{\dz S}$ is a core of $N$; this definition works
regardless what class of operators the extensions belong to, the
only requirement is \eqref{5.28.12} to have sense. Theorem
\ref{t2.3.1} provides us with formally normal extensions of cyclic
type of an operator satisfying ({\tt PD}).

Minimal \underbar{normal} extensions of cyclic type may
\underbar{not} exist -- see \cite{try3} and \cite{deter}, the
latter is for further, much broader development of the former; an
example of another type is in \cite{expl}. In this matter quote
Theorem 3 and Corollary 3, both in \cite{try3}, in one.
   \begin{thm} \tlabel{t2.18.1}
   Let S be a subnormal operator. Suppose that it \underbar{has}
at least one minimal normal extension of cyclic type. Then an
arbitrary normal extension of S is minimal of \underbar{spectral}
type if and only if it is minimal of \underbar{cyclic} type.

  If S has at least one minimal normal extension of cyclic type,
then all its minimal normal extensions of spectral type $($hence
those of cyclic type too$)$ are $\hhc$--equivalent.
   \end{thm}

This settles the question of uniqueness. Let us say carefully that
a subnormal operator $S$ has the {\em uniqueness extension
property} if the circumstances of Theorem \ref{t2.18.1}
happen\,\footnote{\;Uniqueness extension property (and
subnormality itself) has been characterized in \cite{nag_ext},
Theorems 4, 4',5 and 5'. When specialized to the complex moment
problem it matters ultradeterminacy resembling the
characterization of Hamburger of determinacy in the real case, cf.
\cite{sh} p. 70.}.
   Part ($\beta'$) of Theorem \ref{t6.10.1} implies at once
   \begin{cor} \tlabel{t1.11.1}
   $S$ has the uniqueness extension property if its domain is
composed of vectors of ultradeterminacy.
   \end{cor}
   \begin{rem} \tlabel{t1.28.3}
   It is clear that for a cyclic subnormal operator $S$ with a
   cyclic vector $e$ the following statements are equivalent:
   \begin{enumerate}
   \item[${\boldsymbol{\cdot}}$] $e$ is a vector of
   ultradeterminacy of $S$, \item[{${\boldsymbol{\cdot}}$}] the
   formally normal extension $N$ of $S$ constructed via Theorem
   \ref{t2a.3.1} is essentially normal, hence it is minimal of
   cyclic type.
   \end{enumerate}
   The other way around, it seems to be worthy to realize how
minimality of cyclic type can be inherited by cyclic subspaces.
More precisely, If $N$ is a minimal normal extension of $S$ acting
in $\kkc$ then for $e\in\dz S$ the closure $\overline{\ddc_e(N)}$
of ${\ddc_e(N)}$ defined in \eqref{2.29.3} reduces $N$ (indeed,
because $\overline{\ddc_e(N)}$ is invariant for both $N$ and $N^*$
we get it, cf. footnote \ref{t}). Therefore, $N_e$ is a minimal
normal extension of $S_e$ of cyclic type and, consequently, $e$ is
a vector of ultradeterminacy of $S$.
   \end{rem}

    \subsubsection*{Subnormality trough $\ccc^\infty$--vectors}
    The following can be regarded as what corresponds to Halmos'
characterization of subnormality. Notice that his boundedness
condition \eqref{1.3.1} takes now a more subtle form of a growth
condition.
    \begin{thm} \tlabel{t9.3.1}
   If $S$ satisfies {\rm(}{\tt PD}{\rm)} and $\dz
S\in\{\bbc(S),\aac(S),\lin\qqc(S)\}$ then it is subnormal and has
the uniqueness extension property.
   \end{thm}
   We refer to \cite{try1}, \cite{try2} and also to
\cite{book_sub} for proofs. They consist in showing the vectors in
$\dz S$ are in fact vectors of ultradeterminacy of $S$, cf.
Corollary \ref{t1.11.1}.

This result is a sort of standard if one restricts an interest to
essential selfadjointness of symmetric operators.

    \subsubsection*{Complete characterization of subnormality by
positive definiteness} What differs Theorem \ref{t9.3.1} from
Bram's result is some oversupply, the presence of an additional
conclusion. The characterization we give below does not have this
defect. In \cite{polar} one can find actually two kinds of
characterizations: the first makes use of extending the positive
definiteness condition ({\tt PD}) so as to get a spatial
extension, the second is just a test, rather complicated to state
it in this paper. Here we describe the first approach. Instead of
stating it formally we explain the idea behind the result by a
sequence of three drawings, they refer to the cyclic case when
({\tt PD}) can be though of as ({\tt MPD}). Only Picture 2 needs
some comment. It refers to the situation of positive definiteness
({\tt MPD}) defined on $\zzb\times\zzb$. In this case we get a
solution and the extra conclusion that the measure involved does
not have $0$ in its support.

    \begin{center}
   \setlength{\unitlength}{1.5pt}
   \begin{picture}(100,90)(-45,-32) 
   \thinlines
   \put(-30,7){\vector(1,0){70}} 
   \put(5,-25){\vector(0,0){70}} 
   \put(1,1){$\scriptstyle{0}$}
    \put(40,0){$\scriptstyle{m}$} \put(-0.5,40){$\scriptstyle{n}$}

     \multiput(5,31)(6,0){5}{\circle*{1}}
   \multiput(5,25)(6,0){5}{\circle*{1}}
   \multiput(5,19)(6,0){5}{\circle*{1}}
   \multiput(5,13)(6,0){5}{\circle*{1}}
   \multiput(5,7)(6,0){5}{\circle*{1}}
   \end{picture}
    \\[1.5ex]
   {\small Figure 1.\ Halmos positive definiteness: {\em too
little}}

   \end{center}
     \vspace{0.3cm}

   \begin{center}
   %
   \setlength{\unitlength}{1.5pt}
   \begin{picture}(100,90)(-45,-32)
   \thicklines \thinlines
   \put(-30,7){\vector(1,0){70}} 
   \put(5,-25){\vector(0,0){70}} 
   \put(1,1){$\scriptstyle{0}$}
    \put(40,0){$\scriptstyle{m}$} \put(-0.5,40){$\scriptstyle{n}$}
\multiput(-19,31)(6,0){9}{\circle*{1}}
\multiput(-19,25)(6,0){9}{\circle*{1}}
\multiput(-19,19)(6,0){9}{\circle*{1}}
\multiput(-19,13)(6,0){9}{\circle*{1}}
   \multiput(-19,7)(6,0){9}{\circle*{1}}
\multiput(-19,1)(6,0){9}{\circle*{1}}
\multiput(-19,-5)(6,0){9}{\circle*{1}}
\multiput(-19,-11)(6,0){9}{\circle*{1}}
\multiput(-19,-17)(6,0){9}{\circle*{1}}

   \end{picture}
   \\[1.5ex]
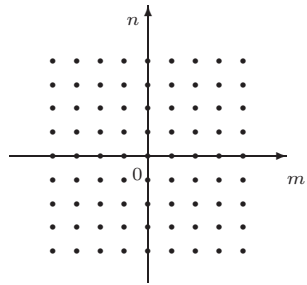
   {\small Figure 2.\ Very extended positive definiteness: {\em
too much}}
   \end{center}
   \vspace{0.3cm}

   \setlength{\unitlength}{1.5pt}
   \begin{center}
   \begin{picture}(100,90)(-45,-32)
   \thicklines \thinlines
   \put(-30,7){\vector(1,0){70}} 
   \put(4.6,-25){\vector(0,0){70}} 
   {\put(-24.5,36){\line(1,-1){58}}} 
   \put(1,1){$\scriptstyle{0}$}
    \put(40,0){$\scriptstyle{m}$} \put(-0.5,40){$\scriptstyle{n}$}
\multiput(-19.4,31)(6,0){9}{\circle*{1}}
\multiput(-13.3,25)(6,0){8}{\circle*{1}}
\multiput(-7.3,19)(6,0){7}{\circle*{1}}
\multiput(-1.3,13)(6,0){6}{\circle*{1}}
   \multiput(4.6,7)(6,0){5}{\circle*{1}}
\multiput(10.4,1)(6,0){4}{\circle*{1}}
\multiput(16.5,-5)(6,0){3}{\circle*{1}}
\multiput(22.6,-11)(6,0){2}{\circle*{1}}
\multiput(28.8,-17)(6,0){1}{\circle*{1}}
   \end{picture}
    \\[1.5ex]
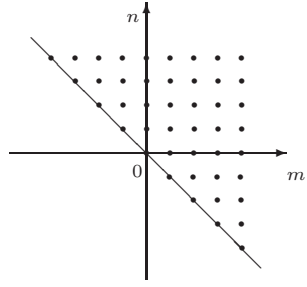
   {\small Figure 3.\ `Half plane' extended positive definiteness:
{\em that's it!}}
   \vspace{0.3cm}
    \end{center}

   \subsection*{The example}   The most spectacular example of the
theory is the creation operator of the quantum harmonic
oscillator; this was notified explicitly for the first time in
\cite{hol}. This operator has many faces. Before we describe them
here let us mention that from the abstract point of view they are
indistinguishable: precisely any of them is the weighted shift
with the weights $\sigma_n=\sqrt{n+1}$ in a particular Hilbert
space and with respect to particular orthonormal basis; the
adjoint acts as a backward shift according the usual rule. The
creation operator is not only so prominent example but also
belongs to the family of the best behaving subnormal operators.
Among the pleasant features of the creation operator $S$ we
mention:
   \begin{enumerate}
   \item[\liczp 1] $\dz {\sbar S}=\dz{S\gw}$, this rounds up
 \eqref{3.2.1};\item[\liczp 2] $S$ has `enough' analytic
 vectors;\item[\liczp 3] $S$ enjoys the uniqueness property;
 \item[\liczp 4] $S$ has a `full' analytic model; \item[\liczp 5]
 it is determined by its selfcommutator.\label{komut}
   \end{enumerate}

The collection of models we are going to present in brief below
shows on how many diverse and concrete ways this abstractly
defined operator can be realized, look also at \cite{how}.
   \subsubsection*{$\llc^2(\rrb)$ model}
   The oldest model of the quantum harmonic oscillator couple, the
creation and the annihilation operator, is
 $$S=\frac 1 {\sqrt 2}(x-\frac \D {\D x}),
\quad S^\times =\frac 1 {\sqrt 2}(x+\frac \D {\D x})$$ considered
in $\llc^2(\rrb)$ with $\dz S=\dz
{S^\times}=\lin(h_n)_{n=0}^\infty$ where $h_n$ is the $n$-th
Hermite function
$$h_n=2^{-n/2}(n!)^{-1/2}\pi^{-1/4}\E^{-x^2/2}H_n$$
with $H_n$, the $n$-Hermite polynomial, defined as
   \begin{equation} \label{1.12.1}
   H_n(x)=(-1)^n\E^{x^2}\frac {\D^{\,n}}{\D x^n}\E^{-x^2}\!.
   \end{equation}

   \subsubsection*{Analytic model: multiplication in the
Segal--Bargmann space}
   An analytic model of the quantum oscillator is in
   $\aac^2(\exp(-|z|^2\D x\!\D y)$, called the Bargmann-Segal
   space -- cf. \cite{segal, barg}, which is composed of all
   entire functions in $\llc^2 (\exp(-|z|^2\D x\!\D y))$ and is,
   in fact, a reproducing kernel Hilbert space with the kernel
   $(z,w)\mapsto\exp(z\bar w)$. The standard orthonormal basis
   $\{e_n\}_{n=0}^\infty$ in the space $\aac^2(\exp(-|z|^2)\D
   x\!\D y)$ is composed of monomials
$$e_n=\frac {z^n}{\sqrt{n!}},\quad z\in\ccb,\quad n=0,1,\dots $$
Set $\ddc_0=\lin(e_n)_{n=0}^\infty$. Then the operators $S$ and $S
^\times $ defined as $$Sf(z)=zf(z), z\in\ccb,\quad S^\times\!
f=\frac \D{\D z}f,\quad f\in\dz S=\dz {S ^\times }=\ddc_0$$ are
the creation and the annihilation operators.

Notice that $\llc^2(\exp(-|z|^2\D x\!\D y)$ is the natural
extension f $\aac^2(\exp(-|z|^2)\D x\!\D y)$ and the creation
operator $S$, which is the operator of multiplication by the
independent variable, extends to the operator which acts in the
same way in the larger space. Because the latter operator is
normal, the creation operator is a {subnormal} operator. The
annihilation operator, is the projection of the operator of
multiplication by $\sbar z$ in $\llc^2(\exp(-|z|^2\D x\!\D y)$ to
the Segal-Bargmann space.

   The unitary equivalence between $\llc^2(\rrb)$ and
$\aac^2(\exp(-|z|^2\D x\!\D y)$ and its inverse can be implemented
by integral transforms, called Bargmann transform, whose kernels
comes from the generating function of the Hermite polynomials, see
\cite{hall} for more details on this and for a little piece of
history. The Bargmann transform can be used also, via Corollary
\ref{t1.31.12}, to argue that the creation in $\llc^2(\rrb)$ is
subnormal, this is parallel to other arguments.
   \subsubsection*{Analytic model: not very classical}
   The Hermite polynomials, defined as in \eqref{1.12.1}, are now considered as
   those in a complex variable. Let $0<A<1$. Then
$$\int_{\rrb^2}H_m(x+\I                  y)H_n(x-\I
y)\exp\big[-(1-A)x^2-\big(\frac 1 A -1\big)y^2\big] \D x\!\D
y=b_n(A)\delta_{m,n}$$ where
$$b_n(A)=\frac    {\pi\sqrt    A}{1-A}\big(2\frac{1+A}{1-A}\big)^n
n!\,.$$ Introducing the Hilbert space $\xxc_A$ of entire functions
$f$ such that
$$\int_{\rrb^2}|f(x+\i  y)|^2\exp\big[Ax^2-\frac  1  Ay^2\big]\D
x\!\D y<\infty$$ and defining
$$h_n^A(z)=b_n(A)^{-1/2}\E^{-z^2/2}H_n(z),\quad z\in\ccb$$
it was shown in \cite{eind} that $\{h_n^A\}_{n=0}^\infty$ is an
orthonormal basis in $\xxc_A$. From the algebraic relation
$H_{n+1}=2zH_n-H_n'$ we get directly
$$\sqrt
{n+1}\,h_{n+1}^A=\sqrt{\frac{1-A}{2(1+A)}}\,[zh_n^A-(h_n^A)'].$$ %
Set $\ddc_A=\lin(h_n^A)_{n=0}^\infty$. The operators $S_A$ and
$S_A^\times$ defined as
   \begin{multline*}
    S_Af(z)=\sqrt{\frac{1-A}{2(1+A)}}\,[zf(z)-f'(z)],\quad
S_A^\times
f(z)=\sqrt{\frac{1+A}{2(1-A)}}\,[zf(z)+f'(z)],\\z\in\ccb,\quad
f\in\ddc_A
   \end{multline*}
are the creation and the annihilation operator in $\xxc_A$, cf.
\cite{anal}.

   It is interesting to notice that this model realizes a kind of
`homotopy' for the quantum harmonic oscillator between the
$\llc^2(\bbc)$ model and that in the space $\aac^2(\exp(-|z|^2)\D
x\!\D y)$ which are both achieved as $0<A<1$ tends to its two
extremities, cf. \cite{anal}.

   \subsubsection*{Discrete model}
    The Charlier polynomials $\{C_n\npo a\}_{n=0}^\infty$, $a>0$,
    are determined by
   $${\E}^{-az}(1+z)^x=\sum_{n=0}^\infty             C_n\npo
a(x)\frac{z^n}{n!}.
   $$
   They are orthogonal with respect to a nonnegative integer
   supported measure according to
   $$
   \sum_{x=0}^\infty C_m\npo a(x)C_n\npo a(x)
   \frac{{\E}^{-a}a^x}{x!}=\delta_{mn}a^nn!, \quad
   m,n=0,1,\dots\;.
   $$
   Define the Charlier functions (or, rather, the {\it Charlier
   sequences}) $c_n\npo a$, $ n=0,1,\dots$ in discrete variable
   $x$ as
   $$
   c_n\npo a(x)=a^{-\frac{n}{2}}(n!)^{-\frac{1}{2}}C_n\npo a(x){\E
   }^{-\frac{a}{2}} a^{\frac{x}{2}} {(x!)^{-\frac{1}{2}}},\quad
   \text{for $x\ge 0$}.
   $$
   As we know from \cite{yet} so defined Charlier sequences
   satisfy
   \begin{equation*}
   \sqrt{n+1} c_{n+1}\npo a(x)=\begin{cases} \sqrt {x} c_n\npo
   a(x-1)-\sqrt a c_n\npo a(x)\;& x\ge 1\\ -\sqrt a c_n\npo a (x)
   \; & x=0\end{cases}
   \end{equation*}
   Therefore, the operator ${S_a}$ defined as
   \begin{equation*}
   (S_a f)(x)\okr
    \begin{cases}
    \sqrt {x} f(x-1)-\sqrt a f(x)\;& x\ge 1\\
   -\sqrt af (x) \; & x=0
   \end{cases}
   \end{equation*}
   with domain $\dz{S_a}\okr\lin\zb{c_n\npo a}{n=0,1,\dots}$ is
the creation. The annihilation operator is defined again as a
finite difference operator
   \begin{equation*}
   (S_a^\times f)(x)= \sqrt{x+1}f(x+1)-\sqrt af(x),\quad
x=0,1,\dots\,
   \end{equation*}
   In \cite{yet} one can find an analog of Bargmann transform for
this model as well.

   \subsubsection*{Plays with the commutation relation}
Remark at \liczp 5, p. \pageref{komut}, has to be developed a
little bit more. It suggests the creation operator is in sense
exceptional. It is clear the creation operator $S$ and its formal
adjoint $S^\times$, the annihilation operator, satisfy the
canonical commutation relation of the quantum harmonic oscillator
   \begin{equation} \label{1.14.1}
   S^\times S-SS^\times=I.
   \end{equation}
   This relation has a rather formal appearance but after giving
it a proper meaning makes the way back possible, cf. \cite{ccr}.
Roughly, an operator $S$ in a separable Hilbert space is a
creation operator if (and only if) it satisfies \eqref{1.14.1}
properly understood, is subnormal and has the uniqueness extension
property.

   Another unprecedented feature the creation operator may be
proud of is that it is uniquely determined as the only operator
within the class of weighted shifts for which its translate(s) is
still there, cf. \cite{pec} and also \cite{charlier} where the
role of the discrete model in $\ell^2$ is fully explained.

   \subsection*{The question}
   It is clear that a `suboperator' of a subnormal operator is by
definition subnormal too. The problem is to what extend the
converse holds true. More specifically,

   \begin{equation}
   \text{if every $S_f$, for $f\in\dz S$, is subnormal, is so $S$?
} \tag{$\clubsuit$}
   \end{equation}
   It is so in ({$\clubsuit$}) for bounded operators, see
\cite{alan, trent}. In the unbounded case this is true if $\dz
S=\aac(S)$, the analytic vectors of $S$, see \cite{ss0}. Replacing
analytic vectors by vectors of determinacy, as in part ($\alpha'$)
of Theorem \ref{t6.10.1} leads to the positive answer provided
\eqref{1.29.3} holds; here extra conclusion appears. However,
condition \eqref{1.29.3} itself is sufficient for ($\clubsuit$) to
be true, see \cite{sesq} and Theorem 4 in \cite{cyk}. It is also
answered in positive when cyclic portions $S_f$ are replace by, so
to speak, $2$--cyclic ones, see \cite{polar} and \cite{cyk}. Our
believe is the problem is a kind of selection one, see \cite{cyk}
for more discussion in this matter. All this supports the
conjecture that it is `yes' at large. Who knows?

   \section*{The end}

   \subsection*{Missing topics}
As always happens when one wants to write a survey of moderate
length and the material is of considerable size the problem of
selection becomes unavoidable. This has happened here as well.
Among the topics which are absent we mention two.
  \subsubsection*{Lifting commutant} The only thing we can do
right now is to direct to \cite{sto_li}, \cite{maj} and
\cite{majsto} where further references can be found.

   \subsubsection*{Analytic models}  Analytic models for unbounded
operators are exhaustively presented in \cite{try3}; their
relation to subnormality is also there. Analytic models are
intimately associated with reproducing kernel Hilbert spaces, cf.
\cite{cue} and \cite{mult}. Let us mention that from this point of
view the question of subnormality can be roughly rephrased as the
problem of integrability of those space. More precisely, when a
reproducing kernel Hilbert space composed of analytic functions
can be isometrically imbedded in an $\llc^2$ space. It is clear
that the Dirichlet space is not such.

   \subsection*{Some final words} This is a  story of unbounded subnormality
as it has been more or less developed until now. This is also an
open invitation to take part in its continuation. {\em
Impressionism} as understood in painting\,\footnote{\; An excerpt
from \cite{imp}: `The style aims at defining a natural yet
visually innovative use of light and color.' } and music at the
turn of the 19th and 20th century does not happen too often in
mathematical writing. Let me keep an `impression' this is my
venture.

   \bibliographystyle{amsplain}

\begin{thebibliography}{99}

   \bibitem{ando} T. And\^o, Matrices of normal extensions of
   subnormal operators, {\it Acta Sci. Math. {\rm(}Szeged{\rm)}},
   {\bf 24}\,(1963), 91--96.

   \bibitem{barg} V. Bargmann, {On a Hilbert space of analytic
functions and an associated integral transform}, {\em Comm. Pure
Appl. Math.}, {\bf 14}(1961), {187--214}.

   \bibitem{cobir} G.Biriuk and E.A. Coddington,
Normal extensions of unbounded formally normal operators. {\em J.
Math. Mech.} {\bf 13} 1964 617–637.

   \bibitem{bir} M.S.    Birman and    M.Z.     Solomjak
{\em Spectral theory of self-adjoint operators in Hilbert space },
{D. Reidel Publishing Company}, {Dordrecht, Boston, Lancaster,
Tokyo}, {\bf 1987}.

   \bibitem{bishop} E. Bishop, Spectral theory of operators  on
a Banach space, {\it Trans. Amer. Math. Soc.}, {\bf 86}(1957),
414-445.

    \bibitem{bra}
   J. Bram, Subnormal operators, {\it Duke Math. J}, {\bf
22}(1955), 75--94,

   \bibitem{brown} A. Brown, {On a class of operators},
{\em Proc. Amer. Math. Soc.}, {\bf 4}\,(1953), {723--728}.

   \bibitem{deter} D. Cicho\'n, Jan Stochel and F.H. Szafraniec,
How much indeterminacy may fit in a moment problem. An example,
submitted.

   \bibitem{cod} E.A.  Coddington, {  Formally
normal operators having no normal extension}, {\em Canad. J.
Math.}, {\bf 17}\,(1965), {1030--1040}.

   \bibitem{codd} \bysame,  Extension theory of formally normal and symmetric
subspaces, {\em Memoirs of the American Mathematical Society}, No.
134. American Mathematical Society, Providence, R.I., {\bf 1973}.


    \bibitem{con} J.B. Conway, {\em The   theory   of
subnormal operators}, Mathematical Surveys and Monographs, {
Providence, Rhode Island}, {\bf 1991}.

   \bibitem{die} J. Dieudonn\'e, {\em Treatise on analysis} vol.
   II, Academic Press, New York and London, {\bf 1970}.

   \bibitem{eind} S.L.L. van  Eijndhoven and J.L.H.  Meyers,
  New orthogonality relations for the Hermite polynomials and
related Hilbert spaces, {\em J. Math. Ann. Appl.}, {\bf
146}\,(1990), 89-98.

   \bibitem{fojasz} C. Foia\c{s}, D\'ecomposition en op\'erateurs et vecteurs
propres. I. \'Etudes de ces d\'ecom\-positions et leurs rapports
avec les prolongements des op\'erateurs, {\it Rev. Roumaine Math.
Pures Appl.}, {\bf 7}\,(1962), {241--282}.


   \bibitem{fri} J.  Friedrich,  A note on the two dimensional
  moment problem, {\it Math. Nach}, {\bf 121}(1985), 285-286.

   \bibitem{fu} {B. Fuglede}, {The multidimensional
moment problem}, {\it Expo. Math.}, {\bf 1}\,(1983), {47--65}.

   \bibitem{hall} B.C. Hall, Holomorphic methods in analysis and mathematical physics,
{\it Contemp. Math.}, {\bf 260}\,(2000), 1-59.

   \bibitem{halmos} P. Halmos, Normal dilations and extensions of
operators, {\em Summa Brasiliensis Math.}, {\bf 2}\,(1950),
125-134.

   \bibitem{henki} S. Hassi, H.S.V. de Snoo and F.H. Szafraniec,
Componentwise decompositions and Cartesian
decompositions of linear relations, submitted.


   \bibitem{hol} A.S. Holevo, {\em Probabilistic and statistical aspects
   of quantum theory}, North-Holland, Amsterdam - New York -
Oxford, {\bf 1982}.

   \bibitem{alan} A. Lambert, Subnormality and weighted shifts,
{\it J. London Math. Soc.}, {\bf 14}(1976), 476--480.

   \bibitem{majsto} W. Majdak, Z.
Sebestyén, J. Stochel and J.E. Thomson, A local lifting theorem
for subnormal operators, {\em Proc. Amer. Math. Soc.} {\bf
134}\,(2006), 1687-1699.

   \bibitem{maj} W. Majdak, A lifting theorem for unbounded
quasinormal operators, {\it J. Math. Anal. Appl.} {\bf
332}\,(2007), 934-946.


   \bibitem{nai} M.A. Naimark, On a representation of additive set
functions, {\em C. R. Doklady Acad. Sci. URSS}, {\bf 41}\,(1943),
359--361.

   \bibitem{nel} E.  Nelson,  Analytic  vectors,  {\it  Ann.
Math.} {\bf 70} (1959), 572-615.

   \bibitem{nuss} A.E.  Nussbaum, {Quasi--analytic
vectors}, {\em Ark. Mat.} {\bf 6}\,(1965). {179--191}.

   \bibitem{sun} G.
McDonald and C. Sundberg, {On the spectra of unbounded subnormal
operators}, {\em Can. J. Math.}, {\bf 38}\,(1986) {1135--1148}.

   \bibitem{ota1} S. \^Ota, {Closed  linear  operators
with domain containing their range}, {\em Proc. Edinburgh Math.
Soc.}, {\bf 27}\,(1984), {229--233}, see also the MathSciNet
review \# MR760619 (86e:47002).

   \bibitem{ota2} \bysame, {On normal extensions of
unbounded operators}, {\em Bulletin of the Polish Acad. Sci.
Math.}, {\bf 46}\,(1998), {291--301}.

   \bibitem{segal}  I.E. Segal, {Mathematical
      problems of relativistic physiscs}, Chap. IV in `Proceedings
of the Summer Seminar', Boulder, Colorado 1960, vol II, ed. M.
Kac; {\it Lectures in Applied Mathematics}, AMS, Providence, RI,
{\bf 1963}.

   \bibitem{sh} J.A. Shohat and J.D. Tamarkin, {\em The
problem of moments}, {Amer. Math. Soc. Colloq. Publ.} , vol. {15},
{Amer. Math. Soc.}, {Providence, R.I.}, {\bf 1943}

   \bibitem{sto_li} Jan Stochel, Lifting strong commutants of
unbounded subnormal operators, {\it Integral Equations Operator
Theory}, {\bf 43}\,(2002) 189–214.
   \bibitem{ss0} Jan Stochel and F.H. Szafraniec,
   A characterization of subnormal operators, {\it Operator
Theory: Advances and Applications}, {\bf 14}(1984), 261-263.

   \bibitem{try1} \bysame, {On normal  extensions  of
unbounded operators. I}, {\em J. Operator Theory}, {\bf
14}\,(1985), {31--55}.
   \bibitem{try2} \bysame, {On normal extensions of
unbounded operators. II}, {\em Acta Sci. Math. {\rm (}Szeged\/{\rm
)}}, {\bf 53}\,(1989), {153--177}.

   \bibitem{try3} \bysame, {On normal extensions  of
unbounded operators. III. Spectral properties}, {\em Publ. RIMS,
Kyoto Univ.}, {\bf 25}\,(1989), {105--139}.

   \bibitem{part} \bysame,   The normal part of an unbounded
   operator, {\it Nederl. Akad. Wetensch. Proc. Ser. A}, {\bf
92}\,(1989), 495-503 = {\it Indag. Math}, {\bf 51}\,(1989),
495-503.

   \bibitem{ass} \bysame, {A few assorted  questions
about unbounded subnormal operators}, {\em Univ. Iagel. Acta
Math.}, {\bf 28}\,(1991), {163--170}.

   \bibitem{polar} \bysame, {The complex moment problem
and subnormality: a polar decomposition approach}, {\it J. Funct.
Anal.}, {\bf 159}\,(1998), {432-491}.

   \bibitem{pec} \bysame,  A peculiarity of the creation operator,
{\it Glasgow Math. J.}, {\bf 44}\,(2001), 137-147.



   \bibitem{book_sub} \bysame, {\em Unbounded operators and
subnormality}, monograph in preparation.

   \bibitem{pams} F.H. Szafraniec,  Dilations on involution semigroups,
{\em Proc. Amer. Math. Soc.}, {\bf 66}\,(1977), 30-32.

   \bibitem{expl} \bysame, {A RKHS of entire functions
and its multiplication operator. An explicit example}, {\em
Operator Theory: Advances and Applications}, {43}, {309-312}, {\bf
1990}.

   \bibitem{sesq} \bysame, Sesquilinear selection of elementary spectral measures and
subnormality, in {\it Elementary Operators and Applications},
Proceedings, Blaubeuren bei Ulm (Deutschland), June 9-12, 1991,
ed. M.Mathieu, pp. 243-248, World Scientific, Singapore, {\bf
1992}.

   \bibitem{nag_ext} \bysame,  The  Sz.-Nagy
"th\'eor\`eme principal" extended. Application to subnormality,
{\it Acta Sci. Math. {\rm(}Szeged{\rm)}}, {\bf 57}\,{(1993)},
{249-262}.

   \bibitem{step} \bysame, A (little) step towards orthogonality of
analytic polynomials, {\em J. Comput. Appl. Math.}, {\bf
49}\,(1993), 255-261.

   \bibitem{yet} \bysame, Yet another face of the creation operator,
 {\em Operator Theory: Advances and Applications}, 80, pp.
266-275, Birkhäuser, Basel, 1995.


   \bibitem{anal} \bysame, Analytic models of the quantum
harmonic oscillator, {\em Contemp. Math.}, {\bf 212}\,(1998), 269-
276.

   \bibitem{ccr}  \bysame,  Subnormality  in   the   quantum
harmonic oscillator, {\it Commun. Math. Phys.}, {\bf 210}\,(2000),
323-334.

   \bibitem{cue} \bysame,  The  reproducing  kernel  Hilbert
space and its multiplication operators, in {\it Operator Theory:
Theory and Applications}, vol. 114, pp. 253-263, Birkh\"auser,
Basel, {\bf 2000}.

   \bibitem{charlier} \bysame,  Charlier polynomials and
translational invariance in the quantum harmonic oscillator, {\em
Math. Nachr.}, {\bf 241}\,(2002), 163-169.

   \bibitem{mult} \bysame,  Multipliers in the reproducing kernel
Hilbert space, subnormality and noncommutative complex analysis,
{\em Oper. Theory Adv. Appl.}, {\bf 143}\,(2003), 313-331.

   \bibitem{how}  \bysame,  How to recognize the creation
operator, {\em Rep. Math. Phys.}, {\bf 59}\,(2007), 401-408.

   \bibitem{try4} \bysame, On
   normal extensions of unbounded operators. IV. A matrix
   construction, {\em Oper. Th. Adv. Appl.}, { 163}, 337-350, {\bf
   2005}.

   \bibitem{cyk} \bysame,   Subnormality and cyclicity,
{\em Banach Center Publications}, {\bf 67}\,(2005), 349-356.

   \bibitem{nagy42} B. Sz.--Nagy,  {\em Spektraldarstellung
linearer Transformationen des Hilbertschen Raumes}, Springer
Verlag, Berlin, {\bf 1942}.

   \bibitem{nagy} \bysame,     {\it Extensions   of
linear transformations in Hilbert space which extend beyond this
space}, Appendix to F. Riesz, B. Sz.--Nagy, {\it Functional
Analysis}, Ungar, New York, {\bf1960}


   \bibitem{nagy_rus} \bysame,     {\it Unitary dilations of Hilbert space operators
and related topics}, Appendix no. 2 to F. Riesz, B. Sz.--Nagy,
{\it Functional Analysis} 2nd Russian edition (extended), Izdat
"Mir", Moscow {\bf1979}.

   \bibitem{wei} J. Weidmann, {\em Linear operators in Hilbert
spaces}, Springer--Verlag, Berlin, Heidelberg, New York, {\bf
1980}.

   \bibitem{trent} T.T. Trent, New conditions for subnormality,
{\it Pacific J. Math.}, {\bf 93}(1981), 459--464.

   \bibitem{imp} http://library.thinkquest.org/C0111578/nsindex.html
   \end{thebibliography}
   
   \end{document}